\documentclass{amsart}

\usepackage[utf8]{inputenc}
\usepackage{amsmath,amssymb,bbm,enumitem,array,multicol}
\usepackage[autolanguage]{numprint}
\usepackage[pdfencoding=auto]{hyperref}
\usepackage{amsthm}
\usepackage{listings}[2013/08/05]
\usepackage[capitalize]{cleveref}
\usepackage[square,numbers,sort]{natbib}

\hypersetup{pdfstartview=}
\newcolumntype{C}{>{$}c<{$}}

\newcommand\BZ{\mathbb Z}
\newcommand\BN{\mathbb N}

\newcommand{\cyc}[1]{\langle #1 \rangle}

\def\BHM#1.#2.#3.#4.{{^{#1}_{#3}\mathcal B^{#2}_{#4}}}

\newcommand\comm\curlyvee
\newcommand\cocomm\curlywedge

\theoremstyle{plain}
\newtheorem{thm}{Theorem}[section]

\newtheorem{cor}[thm]{Corollary}
\newtheorem{prop}[thm]{Proposition}
\newtheorem{lem}[thm]{Lemma}

\theoremstyle{definition}
\newtheorem{df}[thm]{Definition}
\newtheorem{example}[thm]{Example}

\theoremstyle{remark}
\newtheorem{rem}[thm]{Remark}

\crefname{lem}{Lemma}{Lemmas}
\crefname{thm}{Theorem}{Theorems}
\crefname{cor}{Corollary}{Corollaries}
\crefname{prop}{Proposition}{Propositions}
\crefname{example}{example}{examples}
\crefname{df}{Definition}{Definitions}
\crefname{equation}{equation}{equations}

\numberwithin{equation}{thm}

\renewcommand\iff{\Leftrightarrow}

\def\clap#1{\hbox to 0pt{\hss#1\hss}}

\newcommand\inv{^{-1}}

\newlist{lemenum}{enumerate}{1}
\setlist[lemenum]{label=\roman*), ref=\textup{\thethm~(\roman*)}}
\crefalias{lemenumi}{lemma}

\title[$FSZ$ properties and wreath products]{Some behaviors of \texorpdfstring{$FSZ$}{FSZ} groups under central products, central quotients, and regular wreath products}
\author{Marc Keilberg}
\email{keilberg@usc.edu}
\begin{document}
\begin{abstract}
We show that any group $G$ with a non-$FSZ_m$ quotient by a central cyclic subgroup also provides a non-$FSZ_m$ group of order $m|G|$ obtained as a central product of $G$ with a cyclic group.  We then construct, for every prime $p>3$ and $j\in\BN$, an $FSZ_{p^j}$ group $F$ such that there is a central cyclic subgroup $A$ with $F/A$ not $FSZ_{p^j}$.  We apply these results to regular wreath products to construct an $FSZ$ $p$-group which is not $FSZ^+$ for any prime $p>3$.  These give the first known examples of $FSZ$ groups that are not $FSZ^+$.  We are also able to prove a few partial results concerning the $FSZ$ properties for the Sylow subgroups of symmetric groups.  In the appendix we enumerate all non-$FSZ$ groups of order $5^7$.
\end{abstract}
\thanks{The author thanks Geoff Mason for conversations about early drafts of related results, and for suggesting that investigating the Sylow subgroups of symmetric groups would be worthwhile.}
\maketitle
\section*{Introduction}
Recent work of \citet{PS16} and the author \citep{K16:monster} has suggested that the $FSZ$ properties for a (perfect) group are completely determined by the $FSZ$ properties of its Sylow subgroups.  A suggestive point of investigation is to consider this connection for the case of the alternating groups $A_n$, or the symmetric groups $S_n$.  These groups were established to be $FSZ$ by \citet{IMM}.  The Sylow $p$-subgroups of $S_n$ are well-known to be described by direct products of iterated regular wreath products of $\BZ_p$ with itself; see \citep{Rot99} for example.  More generally, the Sylow $p$-subgroups of classical simple groups are often given by iterated wreath products with $\BZ_p$ \citep{Weir:ClassicalSylow}.

This paper was born from efforts to understand how well-behaved the $FSZ$ properties of $p$-groups are with respect to regular wreath products with $\BZ_p$, with the goal of determining if the Sylow subgroups of $S_n$ are always $FSZ$.  While this remains an open question in general, we are able to provide some partial results. In \cref{cor:symm-syl-high} we show that $S_{p^j}$ has an $FSZ_{p^{j-1}}$ Sylow $p$-subgroup, and in \cref{thm:Sp3} we show that the Sylow $p$-subgroup of $S_{p^3}$ is $FSZ$.

Along the way, several examples of bad behavior are discovered, including some new constructions of non-$FSZ$ groups.  We also provide in \cref{prop:wreath-condition} a sufficient condition for the $FSZ$ properties to hold for regular wreath products of the form $D\wr_r\BZ_p$ with $D$ a $p$-group.  Our work on such wreath products culminates in \cref{thm:fsz-not-plus}, which provides an $FSZ$ $p$-group of order $p^{(p+1)^2}$ which is not $FSZ^+$ for any prime $p>3$.

The paper is structured as follows.  In \cref{sec:background} we detail the necessary background material, definitions, and notation.  Then in \cref{sec:central} we establish a connection between the $FSZ$ properties of central products with abelian groups and quotients by central subgroups.  In \cref{sec:family} we construct examples of $FSZ_{p^j}$ groups with non-$FSZ_{p^j}$ central quotients and central products.  \Cref{sec:main} is the main section of the paper, and investigates regular wreath products $D\wr_r\BZ_p$ with $D$ a $p$-group.  These investigations are applied in \cref{sec:not-plus} to construct examples of $FSZ$ groups which are not $FSZ^+$, using the groups constructed in the previous sections.  The proof of the $FSZ$ property for these examples is then adapted to show that the Sylow $p$-subgroup of $S_{p^3}$ is $FSZ$.  The appendix enumerates all isomorphism classes of non-$FSZ$ groups of order $5^7$, as determined with \citet{GAP4.8.4}.

\section{Background and notation}\label{sec:background}
All groups are finite.  For a group $G$ and $x\in G$ we let
\[ o(x) = \mbox{order of the element } x. \]  We define $\BN=\{1,2,3,...\}$ to be the set of positive integers.

Given groups $G,H$, subgroups $A\subseteq Z(G)$ and $B\subseteq Z(H)$, and an isomorphism $\phi\colon A\to B$, let $N=\{(a,\phi(a\inv))\in G\times H \ : \ a\in A\}$.  Then we define the central product $G\ast H$ as the quotient group $(G\times H)/N$.  The direct product arises as the special case where $A,B$ are trivial.  In general the isomorphism class of $G\ast H$ depends on $A,B,\phi$.  The specific choices will either always be made explicit, or will not be important, hence our choice to omit these dependencies from the notation.  In an obvious fashion there is an equivalent definition of central product using generators and relations, which we will also use whenever convenient. For our purposes, we will be primarily interested in the case where $A,B,$ and $H$ are all cyclic.  In this case, to define $G\ast H$ it suffices to state a relation $x^m=z$ where $H=\cyc{x}$, $z\in Z(G)$, and $m\in\BN$.

Given a group $G$ and a prime $p$, the regular wreath product $G\wr_r\BZ_p$ is defined as the semidirect product $G^p\rtimes \BZ_p$, where $\BZ_p$ acts by cyclicly permuting the factors of $G^p$.

The $FSZ$ properties for groups were introduced by \citet{IMM}, and were inspired by invariants of representation categories of semisimple Hopf algebras \citep{KSZ2,NS07a,NS07b}.  These invariants and their generalizations have proven extremely useful in a wide range of areas; see \citep{NegNg16} for a detailed discussion and references.

While multiple characterizations of the $FSZ$ properties exist in the literature \citep{IMM,PS16}, for our purposes these properties are concerned with the following sets.
\begin{df}
  For a group $G$, $m\in\BN$, and $u,g\in G$ we define
  \[ G_m(u,g) = \{ a\in G \ : \ a^m = (au)^m = g\}.\]
\end{df}
\begin{rem}
  We note that in the original definitions for these sets, based on the original formulas for the Frobenius-Schur indicators of the Drinfeld double of $G$ \citep{KSZ2,IMM}, one uses $u\inv$ instead of $u$. In this setting an irreducible representation of this Hopf algebra is parameterized by a pair $(u,\eta)$, with $u\in G$ and $\eta$ an irreducible character of $C_G(u)$ \citep{DPR}.  The isomorphism class of such a representation depends only on the conjugacy class of $u$, and the isomorphism class of $\eta$.  The $m$-th indicator of $(u,\eta)$ is then calculated, in our definition, using the sets $G_m(u\inv,g)$.  In principle, then, there is a difference between using $u$ or $u\inv$ in the definition, and either one can be argued to be the more convenient choice.  However, there is a bijection $G_m(u,g)\to G_m(u\inv,g)$ for all $u,g\in G$ and $m\in\BN$ given by $a\mapsto au$.  Since the $m$-th indicator values depend only on the cardinalities of the sets $G_m(u,g)$ \citep{KSZ2}, rather than the sets themselves, there is in fact no loss of convenience in either setting.
\end{rem}
We then define the $FSZ$ properties as follows.
\begin{df}
  Let $m\in\BN$.  We say the group $G$ is $FSZ_m$ if for any $n\in \BN$ with $(n,|G|)=1$ we have
  $|G_m(u,g)|=|G_m(u,g^n)|$ for all $u,g\in G$.

  We say $G$ is $FSZ$ if it is $FSZ_m$ for all $m$.

  Furthermore, we say that $G$ is $FSZ_m^+$ if $C_G(g)$ is $FSZ_m$ for all $g\in G$.  We say $G$ is $FSZ^+$ if it is $FSZ_m^+$ for all $m$.
\end{df}
The fundamental facts of $FSZ$ and $FSZ^+$ groups can be found in \citep{IMM}, as well as many examples and classes of $FSZ^+$ groups.  We record the facts we will use frequently throughout the paper in the following.

\begin{lem}(\citet{IMM})
  \begin{enumerate}
    \item $G_m(u,g)=\emptyset$ if $u\not\in C_G(g)$, and $G_m(u,g)\subseteq C_G(g)$.
    \item Every group is $FSZ_m$ for $m\in \{1,2,3,4,6\}$.
    \item If $G=H\times K$ is a direct product of groups, then
    \[ G_m((u_H,u_K),(g_H,g_K)) = H_m(u_H,g_H)\times K_m(u_K,g_K).\]
    In particular, $G$ is $FSZ_m$ if and only if both $H$ and $K$ are $FSZ_m$.
    \item Every regular $p$-group is $FSZ^+$.
    \item The irregular $p$-group $\BZ_p\wr_r\BZ_p$ is $FSZ^+$.
    \item There are 32 isomorphism classes of (necessarily irregular) non-$FSZ$ groups of order $5^6$.  They are all non-$FSZ_5$ and necessarily have exponent 25, maximal class, and a center of size 5.
    \item The symmetric and alternating groups are all $FSZ^+$.
    \item A group is $FSZ_m^+$ if and only if for all $n\in\BN$ with $(n,|G|)=1$ and $u,g\in G$ with $[u,g]=1$ the sets $G_m(u,g)$ and $G_m(u,g^n)$ are isomorphic permutation modules for $C_G(u,g)$.
  \end{enumerate}
\end{lem}
Several additional examples of $FSZ^+$ and non-$FSZ$ groups can be found in \citep{K16:p-examples,K16:monster,K,K2,PS16,Etingof:SymFSZ}.  It has been an open question as to whether or not there exists an $FSZ$ group which is not $FSZ^+$.  We will show that such groups exist in \cref{thm:fsz-not-plus}.

\begin{rem}
We wish to note an important contrast between regular $p$-groups and its parent class of $FSZ$ $p$-groups.  It is well-known that a direct product of regular $p$-groups is not necessarily regular, whereas the $FSZ$ properties are preserved by direct products.  On the other hand, every subquotient of a regular $p$-group is again regular, whereas there seems to be little necessary connection between the $FSZ$ properties of a group and its subquotients.  That $S_n$ is $FSZ^+$ for all $n$ and the existence of non-$FSZ$ groups shows that an $FSZ^+$ group can have many non-$FSZ$ subgroups.  Furthermore, any proper subquotient of a non-$FSZ$ group of order $p^{p+1}$ is necessarily $FSZ^+$.  In \cref{sec:central} we will show that an $FSZ$ (indeed, $FSZ^+$ by \cref{cor:fp1-plus}) group can have a non-$FSZ$ quotient, even if we restrict attention to quotients by central subgroups.  This reflects the general difficulty, both theoretical and computational, of working with the $FSZ$ properties.
\end{rem}

\section{Connecting central products and central quotients}\label{sec:central}
Given that the class of $FSZ_m$ groups and the class of $FSZ_m^+$ groups are both closed under direct products for every $m\in\BN$, one might hope that they are also closed under central products, or at least central products with abelian groups.  By \citet{IMM} the classes of $FSZ_m$ and $FSZ_m^+$ groups are trivially closed under central products for any $m\in\{1,2,3,4,6\}$, since these are the same as the class of all groups.  Unfortunately, this closure property turns out to fail for prime powers $m=p^j$ for any prime $p>3$ and $j\in\BN$, and it is our goal to illuminate the process by which we can construct examples demonstrating this failure.

By using iterated central products it is easy to see that if we want to prove any sort of closure of $FSZ$ groups with respect to central products with abelian groups, then it suffices to assume the abelian group is cyclic.  Indeed, for $H=G\ast C$ with $C$ cyclic it suffices to suppose that $[H:G]=p$ is a prime.  We can also observe that, by the isomorphism theorems, to prove any sort of closure of $FSZ$ groups with respect to quotients by central subgroups it would suffice to consider the case where the central subgroup is cyclic of prime order.

We will soon see that the following sets are intimately connected to determining these closure properties, and the failure thereof.
\begin{df}\label{df:trip-sets}
  Let $G$ be a group, and let $u,g\in G$ and $z\in Z(G)$.  For any $m\in\BN$ we define
  \[ G_m(u,g,z) = \{ a\in G \ : \ a^m = z(au)^m = g\}.\]
\end{df}
We have the trivial equality $G_m(u,g,1)=G_m(u,g)$ in the usual sense, and we easily see that $G_m(u,g,z)\subseteq C_G(g)$ and that $G_m(u,g,z)=\emptyset$ if $[u,g]\neq 1$.  We note that the restriction $z\in Z(G)$ can be relaxed to $z\in C_G(g)$ without changing these properties, but for our purposes we will only be interested in the case of $z\in Z(G)$.

We begin by considering quotients by cyclic central subgroups.
\begin{thm}\label{thm:cent-quot-decomp}
  Let $G$ be a group with $1\neq z\in Z(G)$.  Define $A=\cyc{z}$ and $H=G/A$ with $\pi\colon G\to H$ the canonical epimorphism.  For fixed $m\in\BN$, $u,g\in G$ and any $n\in\BN$ with $(n,|G|)=1$ define
  \[ X_m(u,g,n)= \bigcup_{i=1}^{o(z)}\bigcup_{j=1}^{o(z)} G_m(u,z^{ni}g^n,z^j).\]
  Then $H_m(uA,g^nA)=\pi(X_m(u,g,n))$ and $X_m(u,g,n)A = X_m(u,g,n)$.

  As a consequence, $H$ is $FSZ_m$ if and only if for all $u,g\in G$ and $n\in\BN$ with $(n,|G|)=1$ there is a bijection $X_m(u,g,1)\to X_m(u,g,n)$.
\end{thm}
\begin{proof}
  By definitions we have
  \[ aA\in H_m(uA,g^nA) \iff \forall a_1\in A \, \exists a_2,a_3\in A \left( a^m a_1 = (au)^m a_2 = g^n a_3\right).\]
  Multiplying this last identity by $a_1\inv$ and using that $A=\cyc{z}$ and $(n,|G|)=1$ shows that \[aA\in H_m(uA,g^nA)\iff a\in X_m(u,g,n).\] This establishes the first claim.

  It then immediately follows that $X_m(u,g,n)$ is a union of $A$ cosets.  Therefore $|X_m(u,g,n)| = |A| |H_m(uA,g^nA)|$, and the final equivalence is then immediate.
\end{proof}
\begin{rem}
  Observe that the unions defining $X_m(u,g,n)$ are in fact disjoint unions.  In other words, $G_m(u,z^{ni}g^n,z^j)\cap G_m(u,z^{ni'},z^{j'}) = \emptyset$ for $i\not\equiv i' \bmod o(z)$ or $j\not\equiv j'\bmod o(z)$.
\end{rem}

We can also exhibit the following connection to the sets $G_m(u,g,z)$.
\begin{lem}\label{lem:cent-prod-lem}
  Let $G$ be a group, and $K=G\ast C$ where $C=\cyc{x}$ and the central product is given by the identification $x^m=z\in Z(G)$ for some $m\in\BN$.  Then for this $m$, any $u,g\in G$, any $n\in\BN$ with $(n,|K|)=1$ and $0\leq j<m$ we have a bijection
  \[ K_m(ux^j,g^n) \to \bigsqcup_{i=0}^{m-1} G_m(u,z^{ni}g^n,z^{j}).\]
  In particular,
  \[|K_m(ux^j,g^n)| = \sum_{i=0}^{m-1} |G_m(u,z^{ni}g^n,z^{j})|.\]
\end{lem}
\begin{proof}
  By definitions we have for $0\leq i < m$ and $a\in G$ that
  \begin{align*}
    ax^{-ni} \in K_m(ux^j,g^n) &\iff z^{-ni} a^m = z^{-ni}z^{j} (au)^m = g^n\\
    &\iff a^m = z^{j} (au)^m = z^{ni}g^{ni}\\
    &\iff a\in G_m(u,z^{ni}g^n,z^{j}).
  \end{align*}
  The claims now follow.
\end{proof}
\begin{cor}\label{cor:cent-prod}
  Let assumptions be as in the preceding lemma.  If also $o(z)\geq m$ then we have a bijection
  \[ K_m(ux^j,g^n) \to \bigcup_{i=0}^{m-1} G_m(u,z^{ni}g^n,z^{j}).\]
\end{cor}
\begin{proof}
  The assumption on $o(z)$ guarantees, by definition, that the sets in the union are all pairwise disjoint.
\end{proof}
So we see there is a strong connection between the $FSZ$ properties for central products and quotients by suitable central subgroups.
\begin{thm}\label{thm:quot-to-prod}
  Let $G$ be a group with $1\neq z\in Z(G)$.  Set $A=\cyc{z}$, $H=G/A$.  Let $m\in \BN$ be divisible by $o(z)$.  Define $K=G\ast C$, where $C=\cyc{x}$ and the central product is defined by $x^m=z$.

  Then $H$ is $FSZ_m$ if and only if for all $u,g\in G$ and $n\in\BN$ with $(n,|K|)=1$ there is a bijection \[\bigcup_{j=0}^{m-1} K_m(ux^j,g)\to \bigcup_{j=0}^{m-1} K_m(ux^j,g^n).\]

  As a consequence, if $K$ is $FSZ_m$ then $H$ is $FSZ_m$.  Equivalently, if $H$ is not $FSZ_m$ then $K$ is not $FSZ_m$.
\end{thm}
\begin{proof}
  In the case that $o(z)=m$ the claims follow from \cref{thm:cent-quot-decomp,cor:cent-prod}.  When $o(z)$ divides $m$ every set in the disjoint union from \cref{lem:cent-prod-lem} occurs the same number of times, and so when applying this bijection to $\bigcup_{j=0}^{m-1} K_m(ux^j,g)$ we see that we obtain every term in $X_m(u,g,z)$ exactly $m/o(z)$ times.  So the result again follows from \cref{thm:cent-quot-decomp}.
\end{proof}

Thus we can show that the class of $FSZ_m$ groups is not closed under central products by showing that it is not closed under quotients by central subgroups.

We know there are non-$FSZ_5$ groups of order $5^6$ by \citet{IMM}, and \citet{GAP4.8.4} has the groups of order $5^7$ in its \verb"SmallGroups" library.  By eliminating cases where central quotients are necessarily $FSZ_5$, we can use GAP to show that of the \numprint{34,297} groups of order $5^7$, there are exactly 83 (isomorphism classes of) $FSZ$ groups of order $5^7$ admitting a non-trivial cyclic subgroup $A$ such that $G/A$ is not $FSZ_5$.  Indeed, all such groups are also $FSZ^+$.  The ids---the value $n$ in \verb"SmallGroup"($5^7$,n)---for these groups are given in \cref{tab:id-nums}.
\begin{table}[ht]
\centering
\caption{ID numbers of $FSZ$ groups of order $5^7$ with non-$FSZ$ central quotients.}
\label{tab:id-nums}
\begin{tabular}{CCCCCCCC}
348& 350& 352& 354&
383& 388& 391& 394\\
397& 453& 458& 463&
467& 472& 477& 530\\
571& 573& 577& 584&
585& 595& 619& 626\\
638& 639& 640& 641&
644& 645& 646& 647\\
654& 660& 667& 700&
701& 702& 703& 709\\
718& 724& 834& 835&
836& 837& 838& 842\\
843& 844& 845& 846&
850& 851& 852& 853\\
854& 858& 859& 860&
861& 862& 866& 867\\
868& 869& 870& 875&
876& 877& 878& 881\\
884& 887& 890& 893&
896& 911& 916& 921\\
926& 931& 936&
\end{tabular}
\end{table}
The following is a specific example that demonstrates this.
\begin{thm}\label{thm:specific}
  Let $G=\verb"SmallGroup"(5^7,348)$ in GAP.  Then $G$ is $FSZ^+$ and there exists $A\subseteq Z(G)$ with $A\cong\BZ_5$ such that $G/A$ is not $FSZ_5$.  There also exists a cyclic group $C$ of order $25$ and a central product $K=G\ast C$ such that $K$ is not $FSZ_5$.
\end{thm}
\begin{proof}
  The last claim follows from \cref{thm:quot-to-prod}, though it can also be constructed and tested in GAP directly.  We will use \citet{GAP4.8.4} and the \verb"FSZtest" function of \citet{PS16} to establish the claims.

  \begin{lstlisting}
    G := SmallGroup(5^7,348);
    Center(G);
    A := Subgroup(G,[G.6*G.7]);

    H := G/A;

    FSZtest(G);
    FSZtest(H);
  \end{lstlisting}
  The \verb"FSZtest" function demonstrates that $G$ is $FSZ$ as desired, and that $H$ is not $FSZ$, and so must necessarily be non-$FSZ_5$.  That $G$ is in fact $FSZ^+$ can be obtained by running \verb"FSZtest" over all suitable centralizers in $G$.
  \begin{lstlisting}
    cl := RationalClasses(G);;
    cl := List(cl,Representative);;
    cl := Filtered(cl,x->not x in Center(G));;

    ForAll(cl,x->FSZtest(Centralizer(G,x)));
  \end{lstlisting}
  This returns true, which completes the proof.
\end{proof}

\section{A family of \texorpdfstring{$FSZ$}{FSZ} groups with non-\texorpdfstring{$FSZ$}{FSZ} central quotients}\label{sec:family}
In this section we show how to obtain an infinite family of examples similar to the one from \cref{thm:specific}.  For this, we will use the groups $S(p,j)$ from \citep{K16:p-examples}.  It was shown in \citep[Theorem 1.9]{K16:p-examples} that $S(p,j)$ was not $FSZ_{p^j}$ for any prime $p>3$ and $j\in\BN$.  We will construct an $FSZ_{p^j}$ group which has $S(p,j)$ as one of its quotients by a central subgroup.  This construction will strongly mirror the one for $S(p,j)$.

So fix a prime $p$ and $j\in\BN$.  Consider the abelian $p$-group
\[ Q_{p,j} = \BZ_{p^{j+1}} \times \BZ_p^{p^j-1}.\]
Note for $p>2$ that $Q_{p,j}=P_{p,j}\times \BZ_p$ as defined in \citep{K16:p-examples}; $P_{p,j}$ was not explicitly defined for $p=2$ for technical reasons, but otherwise makes sense with $P_{2,1}=\BZ_4$.  Let $Q_{p,j}$ have generators $a_1,...,a_{p^j}$, where $a_1$ has order $p^{j+1}$ and the rest have order $p$.

We define an endomorphism $b_{p,j}$ of $Q_{p,j}$ by
\begin{align*}
  a_1 &\mapsto a_1 a_2\inv\\
  a_k &\mapsto a_k a_{k+1}, \ 2\leq k < p^j\\
  a_{p^j} &\mapsto a_{p^j}.
\end{align*}
We can equivalently write $b_{p,j}$ as a $p^j\times p^j$ lower triangular matrix $B_{p,j}$ which acts from the left on $Q_{p,j}$ in the obvious fashion.  The entries in the first row of $B_{p^j}$ are taken modulo $p^{j+1}$, while all other entries are taken modulo $p$.  We have
\begin{align*}
  B_{p,j} &= \begin{pmatrix}
    1 & 0 & 0 &\cdots & 0 & 0 &0\\
    -1 & 1 & 0 &\cdots & 0 & 0 &0\\
    0& 1 & 1 & \cdots & 0 & 0 & 0\\
    &\vdots&&&&\vdots&\\
    0&0&0&\cdots & 1&1&0\\
    0&0&0&\cdots&0&1&1
  \end{pmatrix}.
\end{align*}
The entries of $B_{p,j}^k$ for $1\leq k < p^j$ are then naturally described by binomial coefficients, or equivalently the entries of Pascal's triangle.
\begin{lem}\label{lem:b-is-aut}
  The endomorphism $b_{p,j}$ is an automorphism of order $p^j$.
\end{lem}
\begin{proof}
  The proof is nearly identical to \citep[Lemma 1.1]{K16:p-examples}, so we only sketch the proof.  By writing $B_{p,j}=I+S$ where $I$ is the identity matrix, we can use the binomial theorem to expand \[(I+S)^{p^j} = I + S^{p^j} + \sum_{k=1}^{p^{j}-1} \binom{p^j}{k} S^k.\]  The powers of $S$ are easily computed, and $S$ is readily observed to be nilpotent with $S^{p^j}=0$.  Moreover, since each binomial coefficient in the summation is a multiple of $p$ and the first row of $S$ is the zero vector, we conclude that every term in the summation is the zero matrix.  Finally, the first column of every power $B_{p,j}^i$ for $1\leq i < p^j$ contains an entry of $-1$, and is $B_{p,j}^i$ therefore not the identity matrix.  This completes the proof.
\end{proof}

\begin{df}
  Let $p$ be a prime and $j\in\BN$.  Define a group of order $p^{p^j+2j}$ by
  \[ F(p,j) = Q_{p,j}\rtimes \cyc{b_{p,j}}.\]
\end{df}
By considering the (right) $1$-eigenvectors of $B_{p,j}$ it is not difficult to see that $Z(F(p,j)) = \cyc{a_1^{p}, a_{p^j}}\cong \BZ_{p^j}\times\BZ_p$.  We identify $Q_{p,j}$ and $\cyc{b_{p,j}}$ as subgroups of $F(p,j)$ in the usual fashion.

We then relate the group $F(p,j)$ to $S(p,j)$ in the following manner.
\begin{prop}\label{prop:has-bad-quot}
  Suppose $p$ is an odd prime, and let $G=F(p,j)$ be as above. Define $A=\cyc{a_1^{p^j} a_{p^j}}\subseteq Z(G)$.  Then $G/A\cong S(p,j)$, as defined in \citep{K16:p-examples}.
\end{prop}
\begin{proof}
  Taking $Q_{p,j}$ modulo $A$ yields $P_{p,j}$ in a natural way, and rewriting the definitions of our $b_{p,j}$ modulo $A$ yields exactly the automorphism of $P_{p,j}$ that defines $S(p,j)$ (in the same natural way).
\end{proof}
From now on, whenever convenient we will omit the subscripts from $b_{p,j}$, $Q_{p,j}$, and $B_{p,j}$.

To describe $p^j$-th powers in $F(p,j)$, we need to consider the sums over all elements of $\cyc{B^{p^l}}$ for any $0\leq l \leq j$ and their action on $Q$.  By definition this matrix sum is independent of the choice of generator of the cyclic group in question.  Explicitly, we define for $0\leq l \leq j$
\begin{align} X_{p,j}(p^l) = \sum_{t=1}^{p^{j-l}} B^{tp^l}.\end{align}
Note that $X_{p,j}(p^j)$ is the identity matrix.

\begin{rem}
  The $X_{p,j}(p^l)$ play the same role that the $Y_{p,j}(p^l)$ played in \citep{K16:p-examples}.  The next few lemmas mirror the corresponding lemmas for the $Y_{p,j}(p^l)$.
\end{rem}

\begin{lem}\label{lem:x-1}
  The matrix $X_{p,j}(1)$ is given by
  \[ X_{p,j}(1).\prod_i a_i^{n_i} = a_1^{p^j n_1} a_{p^j}^{-n_1}.\]
\end{lem}
\begin{proof}
  Let $X$ be shorthand for $X_{p,j}(1)$.  Now $X$ satisfies the identity $BX=X$, meaning that the columns of $X$ are 1-eigenvectors of $B$.  These are easily seen to be multiples of $(1,0,...,0,-1)$.  It therefore suffices to show that all columns but the first of $X$ must be zero.  For this, note that $X$ also satisfies $XB=X$, so that the row vectors of $X$ are left 1-eigenvectors of $B$.  These are easily seen to all be multiples of $(1,0,...,0)$, which gives the desired claim.
\end{proof}

\begin{lem}\label{lem:x-L}
  The matrices $X_{p,j}(p^l)$ for $1\leq l \leq j$ all satisfy
  \[ p^l X_{p,j}(p^l).\prod_i a_i^{n_i} = a_1^{p^j n_l}.\]
\end{lem}
\begin{proof}
  Let $X$ be shorthand for $X_{p,j}(p^l)$.  Since $l>0$ the multiplication by $p^l$ guarantees that all entries of $p^l X$, except possibly those in the first row, are zero.  The (1,1) entry of every matrix in the sum defining $X$, on the other hand, is $1$.  Since there are $p^{j-l}$ terms in the summation, the $(1,1)$ entry of $p^l X$ is therefore $p^j$.  Since the only non-zero entry in the first row of $B$ is the (1,1) entry, every other entry on the first row in every power of $B$ is automatically 0.  This completes the proof.
\end{proof}

%

We can then describe the $p^j$-th powers in $F(p,j)$ by
\begin{align}\label{eq:gen-powers}
  \left( a_1^{n_1}\cdots a_{p^j}^{n_{p^j}} b^{p^l r} \right)^{p^j}  &= p^{l}X_{p,j}(p^l). \left(a_1^{n_1}\cdots a_{p^j}^{n_{p^j}}\right),
\end{align}
where $p\nmid r$.

\begin{lem}\label{lem:pj-pows}
  For $a=a_1^{n_1}\cdots a_{p^j}^{n_{p^j}} b^t\in F(p,j)$  we have
  \begin{align*}
    a^{p^j} &= \begin{cases}
      a_1^{n_1 p^j},& p\mid t\\
      a_1^{n_1 p^j} a_{p^j}^{-n_1},& p\nmid t
    \end{cases}.
  \end{align*}
\end{lem}
\begin{proof}
  Follows directly from \cref{lem:x-1,lem:x-L} applied to \cref{eq:gen-powers}.
\end{proof}
\begin{cor}\label{cor:exp}
  $F(p,j)$ has exponent $p^{j+1}$.
\end{cor}

\begin{thm}\label{thm:fpj-main}
  Let $F(p,j)$ be as above.  Then the following all hold.
  \begin{enumerate}
    \item $F(p,j)$ is $FSZ_{p^j}$.
    \item If $p>3$, then $F(p,j)$ has a central subgroup $A\cong \BZ_p$ such that $F(p,j)/A$ is not $FSZ_{p^j}$.
    \item If $p>3$, then for a cyclic group $C$ of order $p^{j+1}$ there exists a central product $H=F(p,j)\ast C$ such that $H$ is not $FSZ_{p^j}$.
  \end{enumerate}
\end{thm}
\begin{proof}
  Set $G=F(p,j)$. The last two parts follow immediately from \citep[Theorem 1.9]{K16:p-examples}, \cref{prop:has-bad-quot,thm:quot-to-prod}.

  For the first part, let $1\neq g=a_1^{k_1}\cdots a_{p^j}^{k_{p^j}} b^r$, $u = a_1^{j_1}\cdots a_{p^j}^{j_{p^j}}b^s$, and $n\in\BN$ with $p\nmid n$.  From \cref{lem:pj-pows} we see that
  \[ G_{p^j}(u,g) = \emptyset\]
  if $g\not\in \cyc{a_1^{p^j},a_{p^j}}\subseteq Z(G)$.

  Let $V$ be the subgroup of $Q$ generated by $a_2,...,a_{p^j}$.  By \cref{lem:pj-pows} the $p^j$-th powers do not depend on elements in $V$.  As such we are free to suppress all elements of $V$ when taking $p^j$-th powers.

  Letting $a= a_1^{n_1}v b^t$ for some $v\in V$, we then have
  \[ a^{p^j} = (a_1^{n_1} b^t)^{p^j},\]
  and
  \[ (au)^{p^j} = (a_1^{n_1+j_1} b^{t+s})^{p^j}.\]
  By \cref{lem:pj-pows} we see that for equality to hold with $1\neq g$ we must have that either $t$ and $t+s$ are both invertible modulo $p$, or both are divisible by $p$. From this it is then easy to see that $|G_{p^j}(u,g)| = |G_{p^j}(u,g^n)|$ for any $p\nmid n$.

  This shows that $F(p,j)$ is $FSZ_{p^j}$ as desired, and so completes the proof.
\end{proof}

\begin{example}
  The group $F(5,1)$ is exactly \verb"SmallGroup"($5^7$,654) in GAP.  This group was the basis for the above construction, in much the same fashion that $S(5,1)=$ \verb"SmallGroup"($5^6$,632) was the motivating example in \citep{K16:p-examples}.
\end{example}
\begin{example}
  Since $p^{p+1}$ is the smallest possible order for any non-$FSZ$ $p$-group, for $p>3$ we see that $F(p,1)$ has the smallest possible order of any $FSZ$ $p$-group with a quotient that is not $FSZ$.  We do not know if $S(p,j)$ has minimal order amongst the non-$FSZ_{p^j}$ $p$-groups. We therefore do not know if $F(p,j)$ has the smallest possible order for an $FSZ_{p^j}$ $p$-group with a (central) quotient that is not $FSZ_{p^j}$.  The minimal orders for, and even the existence of, non-$FSZ$ 2-groups and 3-groups remains an open question.
\end{example}
\begin{rem}
  The groups $F(p,j)$, much like the groups $S(p,j)$, were designed so that there was a simple formula for $p^j$-th powers.  However, lower powers are in general more complex. In particular, these will depend on the $t$ in $b^t$, rather than just the order of $b^t$.  As our ultimate goal of finding an $FSZ$ group which is not $FSZ^+$ will only need to consider $F(p,1)$, we will make no attempt here to understand the $FSZ_{p^l}$ properties of $F(p,j)$ for $l<j$.
\end{rem}

\section{Regular wreath products with \texorpdfstring{$\BZ_p$}{Z/pZ}}\label{sec:main}
Our goal now is to use the results of the preceding section to help us construct, for every prime $p>3$, a group of order $p^{(p+1)^2}$ which is $FSZ$ but not $FSZ_p^+$.  These will yield the first known examples that demonstrate that the $FSZ^+$ properties are strictly stronger then the $FSZ$ properties in general.

The idea is to look for a group $G$ constructed from $F(p,1)$.  We desire that the centralizers in $G$ are constructed from those in $F(p,1)$ in a nice fashion, and in particular that for some $g\in G$ we have $C_G(g)\cong F(p,1)\ast \BZ_{p^2}$ is not $FSZ_p$, but $G_p(u,g^n)=\emptyset$ for all $u\in G$ and $(n,|G|)=1$.  We will see that the regular wreath product $F(p,1)\wr_r \BZ_p$ gives exactly what we want.  First, we will need to review such regular wreath products, and investigate how they behave with respect to the $FSZ$ properties.

For wreath products of the form $D\wr_r \BZ_p$ we write the typical element as \[(d_0,...,d_{p-1},i),\] where $d_0,...,d_{p-1}\in D$ and $i$ is an integer.  The last coordinate and the indices for the first $p$ coordinates are all taken modulo $p$.

\begin{lem}\label{lem:cents}
  Let $D$ be a group, $p$ a prime, and set $G=D\wr_r \BZ_p$.  Then the following hold.
  \begin{lemenum}
    \item \label{lem-part:cents-1}If $i\not\equiv 0\bmod p$, then any element of the form $(d_0,...,d_{p-1},i)$ is conjugate to $(x,1,...,1,i)$ where $x$ is any element of the conjugacy class of $d_0\cdots d_{p-1}$ in $D$.  Letting $\Delta C_D(x)$ be the image of the diagonal embedding of $C_D(x)$ into $G$, we have
        \begin{align*}
            C_G((x,1,...,1,i)) &= \cyc{\Delta C_D(x),(x,1,...,1,i)}\\
            &= \Delta C_D(x)\ast\cyc{(x,1,...,1,i)}
        \end{align*}
        is a central product with respect to the central subgroup $\cyc{x}$.  In particular, $\Delta C_D(x)$ is a normal subgroup of index $p$, whose coset representatives can be taken as the first $p$ powers of $(x,1,...,1,i)$. Indeed, $\Delta C_D(x)$ commutes elementwise with these representatives.
    \item \label{lem-part:cents-2}Any element of the form $(d_0,...,d_{p-1},0)$ such that the $d_i$ are all conjugate to each other is conjugate to $(d,...,d,0)$ where $d$ is any element conjugate to $d_0$ in $D$.  In this case, we have
        \[ C_G(d,...,d,0) = C_D(d)\wr_r \BZ_p.\]
    \item \label{lem-part:cents-3}Any element of the form $(d_0,...,d_{p-1},0)$ such that $d_i$ and $d_j$ are not conjugate for some $i,j$ is conjugate to $(x_0,...,x_{p-1},0)$, where $x_i$ is any element conjugate to $d_i$ in $D$.  In this case, we have
        \[ C_G(d_1,..,d_p,0) = C_D(d_1)\times\cdots\times C_D(d_p).\]
  \end{lemenum}
  Moreover, these enumerate all of the conjugacy classes of $G$.  In particular, no element satisfying one of the cases is conjugate to any element from one of the other cases, or to an element with a different last coordinate.
\end{lem}
\begin{proof}
  All claims are proven by standard manipulations in wreath products, so we sketch only the proof of the first item.  In all cases the result mirrors the determination of the conjugacy classes for iterated wreath products of cyclic groups given by \citet{Orellana2004}.

  Let $i\not\equiv 0\bmod p$, and let $g=(d_0,...,d_{p-1},i)\in G$.  Set $x=d_0\cdots d_{p-1}$ and define $h_0,...,h_{p-1}$ by
  \[ h_{ki} = \Big(\prod_{l=k}^{p-1} d_{li}\Big)\inv, \ 0 < k < p\]
  and $h_0=1$.  Then for $h=(h_0,...,h_{p-1},0)$ we have $g^h = (x,1,...,1,i)$, as desired.  It is clear that $\cyc{\Delta C_D(x),(x,1...,1,i)}\subseteq C_G((x,1,...,1,i))$.  A straightforward check of the commutation relation $g^y=g$ for $y=(y_0,...,y_{p-1},j)$ then gives the reverse equality.
\end{proof}

We will also use the following standard lemma, the proof of which is elementary.
\begin{lem}\label{lem:order-irrev}
  Let $G$ be any group and suppose $g_0,...,g_n\in G$ are such that $g_0,...,g_n\in C_G(g_0\cdots g_n)$.  Then $g_0\cdots g_n = g_{i}\cdots g_{n+i}$ for all $i\in\BZ$, where indices are taken modulo $n+1$.  Equivalently, the product is invariant under a cyclic permutation of the terms.
\end{lem}

For notational convenience, when investigating the $FSZ$ properties for regular wreath products $D\wr_r \BZ_p$ we will sometimes write a typical element $(x_0,...,x_{p-1},j)$ in the short-hand form $(x_l,j)$.  As the main example, if $x_l = \prod_{s=0}^{p-1} y_{l+sj}$ for some $j$ and all $l$, instead of writing $(\prod_{s=0}^{p-1} y_{sj},\prod_{s=0}^{p-1} y_{1+sj},...,k)$, we write simply $(\prod_{s=0}^{p-1} y_{l+sj},k)$.

\begin{thm}\label{prop:wreath-condition}
  Let $D$ be an $FSZ_{p^t}$ and $FSZ_{p^{t-1}}$ group for some $t\in\BN$.  If for any $d,u_0\in D$ the number of elements $(x_0,...,x_{p-1})\in D^p$ satisfying
  \begin{align}\label{eq:wreath}
    x_l^{p^t} = \Big(u_{0} x_1 x_2\cdots x_{p-1} x_0\Big)^{p^{t-1}} = d^n, \ \text{for all } l,
  \end{align}
  does not depend on $n$ when $p\nmid n$, then $D\wr_r\BZ_p$ is $FSZ_{p^t}$.
\end{thm}
\begin{proof}
  We have that $\exp(G)=p\cdot \exp(D)$, so it suffices to consider the $FSZ_{p^j}$ properties for $p^j\mid \exp(D)$.  It is clear that all $p$-th powers in $G$ are elements of $D^p$, whence $G_{p^t}(u,g)=\emptyset$ whenever $g\not\in D^p$.  By assumptions on $D$ it then follows from \cref{lem:cents} that we need only consider the case where $g=(d,...,d,0)$ for some $d\in D$.

  Fix $n\in\BN$ with $p\nmid n$.

  We now proceed to investigate the identities $a^{p^t} = (au)^{p^t} = g^n = (d^n,0)$.  Let $a=(x_l,i)$, $u = (u_l,j)$.  Then
  \begin{align}\label{eq:ab-pt-a-pow}
      a^{p^t} &= \begin{cases}
        ( \Big(\prod_s x_{l+si}\Big)^{p^{t-1}},0),& p\nmid i\\
        (x_l^{p^t},0),& p\mid i
      \end{cases},\\
      (au\inv)^{p^t} &= \begin{cases}
        ( \left(\prod_s x_{l+s(i+j)} u_{l+i+(i+j)s}\right)^{p^{t-1}},0),& p\nmid i+j\\
        ( (x_l u_{l+i})^{p^t},0),& p\mid i+j
      \end{cases}.\label{eq:ab-pt-au-pow}
  \end{align}

  So we have a number of natural cases to consider.

  Suppose first that $p\mid i$ and $p\mid i+j$ (so that $p\mid j$).  Then the equations take the form
  \[ x_l^{p^t} = (x_l u_l)^{p^t} = d^n\]
  for all $l$. The number of solutions to this is independent of $n$ for all $l$ since $D$ is $FSZ_{p^t}$ by assumption.

  Next suppose that $p\nmid i$ and $p\nmid i+j$.  Then the equations take the form
  \[ \left(\prod_s x_{l+si}\right)^{p^{t-1}} = \left(\prod_s x_{l+s(i+j)} u_{l+i+s(i+j)}\right)^{p^{t-1}} = d^n.\]
  So suppose we have any arbitrary elements $x_1,...,x_{p-1}$, and then rewrite the above for $l=0$ as
  \[ (x_0 b)^{p^{t-1}} = (x_0 b v)^{p^{t-1}} = d^n,\]
  with
    \begin{align*}
        b = & \prod_{k=1}^{p-1} x_{ki}\\
        v = & b\inv u_{i}\prod_{k=1}^{p-1} x_{(i+j)k} u_{i+(i+j)k}.
    \end{align*}
  It follows from \cref{lem:order-irrev} that we have a bijection from $D_{p^{t-1}}(v,d^n)$ to the $x_0$ value of those $(x_0,...,x_{p-1},i)\in G_{p^t}(u,g^n)$ with the given $x_1,...,x_{p-1}$ and $i$ via $a\mapsto ab\inv$.  Since $x_1,...,x_{p-1}$ were arbitrary and $u$ is fixed, $v$ does not depend on $n$ and $|D_{p^{t-1}}(v,d^n)|$ does not depend on $n$ by assumptions.  Therefore the number of solutions in this case is also independent of $n$.

  Note that the cases so far combine to completely cover the case where $p\mid j$.  So in the remainder of the cases we have $p\nmid j$.  Now $u\in C_G(g) = C_D(d)\wr_r\BZ_p$ and by \citep[Proposition 7.2]{KSZ2} we always have a bijection $G_m(x,y)\to G_m(x^t,y^t)$ given by $a\mapsto a^t$ for any group $G$, $m\in\BN$ and $x,y,t\in G$.  Therefore by \cref{lem-part:cents-1} we may suppose without loss of generality that $u=(u_0,1,...,1,j)$.

  For the third case, we suppose that $p\mid i$ and $p\nmid i+j$ (and so $p\nmid j$).  Then the equations take the form
  \[ x_l^{p^t} = \left(\prod_s x_{l+sj} u_{l+sj}\right)^{p^{t-1}} = d^n\]
  for all $l$.  By \cref{lem:order-irrev} and assumptions on $u$ the middle term is necessarily always equal to
  \[\Big(u_0\prod_{s=1}^p x_s \Big)^{p^{t-1}},\]
  or equivalently it is equal to
  \[\left(\Big(\prod_{s=1}^{p} x_s\Big)u_0\right)^{p^{t-1}}.\]
  Therefore these equations are of the form given in the statement of the result.  So we continue on to the next case.

  So consider the last case: $p\nmid i$ and $p\mid i+j$.  Then the equations take the form
  \[ \left(\prod_s x_{l+si}\right)^{p^{t-1}} = (x_l u_{l+i})^{p^t} = d^n.\]
  Making the substitution $y_l = x_l u_{l+i}$ and $v_l = u_{l+i}\inv$, the equations become
  \[ y_l^{p^t} = \left(\prod_s y_{l+si} v_{l+si}\right)^{p^{t-1}} = d^n,\]
  which is same form as the equations from the previous case.  Defining $v = (v_l,-j)$, these equations in fact come from exactly the previous case when replacing $u$ by $v$ and the $x_l$ by the $y_l$, so the same manipulations and simplifications hold to get them into exactly the form from the statement.

  This completes the proof.
\end{proof}
\begin{rem}
  If, for fixed $d^n$, the number of solutions to the third and fourth cases in the preceding proof are equal, then we have the converse statement: $D\wr_r\BZ_p$ is $FSZ_{p^t}$ implies that the number of solutions to \cref{eq:wreath} does not depend on $n$ when $p\nmid n$.  Moreover, we can easily see that in the final case of the proof we have $v\inv=u$, and the bijection $G_m(u,g^n)\to G_m(u\inv,g^n)$ given by $a\mapsto au$ preserves the first two cases but swaps the last two cases.  Nevertheless, there seems to be no clear reason for the third and fourth cases to always have the same number of solutions.  In particular, it seems possible that $D$ and $D\wr_r\BZ_p$ are both $FSZ$ but nevertheless fail to satisfy the theorem.  Similarly, it seems possible for $D$ to be $FSZ$ but yet for $D\wr_r\BZ_p$ to be non-$FSZ$.  The author was unable to find any examples demonstrating either behavior, however.  This is in large part due to the prohibitively large number of elements to check even when $D$ has relatively small order.  All such $D$ and $p$ with $|D\wr_r\BZ_p|\leq 50,000$ were verified to be $FSZ$ with \citet{GAP4.8.4}---most of which are 2-groups with $|D|=128$, or trivially $FSZ^+$ by the exponent criterion of \citep[Corollary 5.3]{IMM}.
\end{rem}
\begin{cor}\label{cor:wreath}
  Let $D$ be a $p$-group.  If $D\wr_r\BZ_p$ is $FSZ_{p^j}$ then $D$ is $FSZ_{p^j}$.

  If we suppose $\exp(D)=p^j$ then $D\wr_r\BZ_p$ is $FSZ_{p^j}$ if and only if $D$ is $FSZ_{p^{j-1}}$.
\end{cor}
\begin{proof}
  The first statement follows from \cref{lem:cents}, while the second follows from \cref{prop:wreath-condition}.
\end{proof}
This suffices to obtain a partial result on the $FSZ$ properties of the Sylow subgroups of symmetric groups.
\begin{thm}\label{cor:symm-syl-high}
  Let $j\in\BN$, and let $P$ be a Sylow $p$-subgroup of $S_{p^j}$.  Then $P$ is $FSZ_{p^{j-1}}$.
\end{thm}
\begin{proof}
  We proceed by induction on $j$.  The case $j=1$ has $P\cong \BZ_p$ and the case $j=2$ has $P\cong \BZ_p\wr_r\BZ_p$, both of which are $FSZ^+$ \citep{IMM}.  Suppose the result holds for some $j\geq 2$.  For the case $j+1$ we have $P\cong T\wr_r\BZ_p$ where $T$ is (isomorphic to) a Sylow $p$-subgroup of $S_{p^{j}}$. Now $T$ has exponent $p^{j}$, so the inductive hypothesis and \cref{cor:wreath} shows that $P$ is $FSZ_{p^{j}}$, as desired.  This completes the proof.
\end{proof}

It seems difficult to establish the condition from \cref{prop:wreath-condition} in general, as this seems to require \textit{ad hoc} methods for each choice of $D$.  However, the case where $D$ is an abelian $p$-group is relatively simple.  The following result generalizes \citep[Example 4.4]{IMM}.

\begin{thm}\label{thm:ab-wreath}
  Let $A$ be an abelian $p$-group.  Then $G= A\wr_r\BZ_p$ is $FSZ^+$.
\end{thm}
\begin{proof}
    Since all abelian groups are $FSZ^+$, to show $G$ is $FSZ$ we need only show that the conditions from \cref{prop:wreath-condition} hold.  Moreover, by \cref{lem:cents}, and the fact that central products of (two) abelian groups are again abelian, the only centralizers in $G$ which are not $FSZ$ by assumptions are equal to $G$.  So the result follows as soon as we show that $G$ is $FSZ$.

    Let $A=\bigoplus_i \cyc{a_i}$.  Fix $t\in\BN$ as in \cref{prop:wreath-condition}.  For $x_l,d,u_{0}\in A$, since $d$ is a $p^t$-th powers in $A$ we may write
    \begin{align*}
     x_l &= \prod_k a_k^{n_{l,k}};\\
     d&=\prod_k a_k^{d_k p^t};\\
     u_{0}&= \prod_k a_k^{m_k}.
    \end{align*}
    Since $A$ is abelian \cref{eq:wreath} for $n=1$ becomes
    \[ x_l^{p^t} = \Big( \prod x_s^{p^{t-1}}\Big) u_0^{p^{t-1}} = d.\]
    Since $A$ is abelian, it suffices to consider the power on a single generator $a_k$ at a time.  Indeed, without loss of generality we may assume that $a_k$ has order greater than $p^t$.  The equality $x_k^{p^t}= d$ then becomes $p^t n_{l,k} \equiv p^t d_k \bmod o(a_k)$, which is equivalent to $n_{l,k}\equiv d_k \bmod o(a_k)/p^t$.  So we may write $n_{l,k} = d_k + y_{l,k}\cdot o(a_k)/p^t$ for some integer $y_{l,k}$.  We can do this for all $l$ and any such $k$.  The equality
    \[ \Big(\prod_s x_s^{p^{t-1}}\Big) u_0^{p^{t-1}} = d \]
    becomes
    \[ m_k+\sum_l y_{k,l}\cdot o(a_k)/p \equiv 0 \bmod o(a_k)\]
    for the given (but arbitrary) $k$.

    It follows that the map
    \[ n_{l,k}=d_k + y_{l,k}\cdot o(a_k)/p^t\mapsto n d_k + y_{l,k}\cdot o(a_k)/p^t\]
    yields the necessary bijection between solutions for any $n\in\BN$ with $p\nmid n$.

    This completes the proof.
\end{proof}

\section{An \texorpdfstring{$FSZ$}{FSZ} but not \texorpdfstring{$FSZ^+$}{FSZ+} group}\label{sec:not-plus}
We are now prepared to investigate the $FSZ$ properties of $F(p,1)\wr_r \BZ_p$. By \cref{lem:cents} we will need to know the centralizers of $F(p,1)$.

\begin{lem}\label{lem:Fp1-cents}
  Let $G=F(p,1)$.  Then the centralizers of $G$ are given as follows.
  \begin{enumerate}
    \item For any $g\in Z(G)$, $C_G(g) = G$.
    \item For $g\in Q$ with $g\not\in Z(G)$, $C_G(g)=Q$.
    \item For all other $g\in G$, $C_G(g) = \cyc{g, a_1^p, a_p} =\cyc{g, Z(G)}$.
  \end{enumerate}
  In particular, the centralizer of every non-central element in $G$ is abelian.
\end{lem}
\begin{proof}
  The first statement is trivial.  For the second, clearly $Q\subseteq C_G(g)$.  Moreover, the only elements of $Q$ centralized by a non-trivial power of $b$ are the elements of $Z(G)$, which gives the reverse inclusion.

  For the final statement, it suffices to consider the case with $g=bq$ for some $q\in Q$.  Note that for every $s$ with $p\nmid s$ we have $g^s = b^s q'$ for some $q'\in Q$.  So suppose that $b^s r\in C_G(g)$ for some $r\in Q$ and any fixed $s$. Then $bqb^s r = b^sr b q$ if and only if $(B^s q)q\inv = (Br)r\inv$.  Now for any $s$ (even $p\mid s$), the map $x\mapsto (B^s x)x\inv$ is a group endomorphism of $Q$, and the kernel of this map for $p\nmid s$ is easily seen to be $Z(G)$.  We conclude that $b^s r = (bq)^s \cdot z$ for some $z\in Z(G)$, from which the desired claim follows.
\end{proof}
The result is false for $F(p,j)$ with $j>1$, as then there are non-central elements with non-abelian centralizers.  In particular, the centralizer of $b^p$ will be non-abelian.

\begin{cor}\label{cor:fp1-plus}
  $F(p,1)$ is $FSZ^+$ for any odd prime $p$.
\end{cor}
\begin{proof}
  By \cref{lem:Fp1-cents,thm:fpj-main} all centralizers are $FSZ$, which means $F(p,1)$ is $FSZ^+$ by definition.
\end{proof}
We can now completely determine the $FSZ$ properties of $F(p,1)\wr_r\BZ_p$.

\begin{thm}\label{thm:higher-wreath-FSZ}
  Let $G=F(p,1)\wr_r\BZ_p$.  Then $G$ is $FSZ_{p^t}^+$ for all $t>1$.
\end{thm}
\begin{proof}
  Let $H=F(p,1)$.  We have that $\exp(H)=p^2$, so that $\exp(G)=p^3$.  So we need only consider the case $t=2$.  Since $H$ is $FSZ$, by \cref{cor:wreath} we conclude that $G$ is $FSZ_{p^2}$.  Moreover, by \cref{lem:Fp1-cents,lem:cents,thm:ab-wreath} we conclude that all proper centralizers that are not described by central products are $FSZ_{p^2}$.  By \cref{lem:Fp1-cents,cor:exp} the centralizers that are described by central products have exponent $p^2$ and so are also $FSZ_{p^2}$.

  This completes the proof.
\end{proof}

\begin{thm}\label{thm:fsz-not-plus}
  Let $G=F(p,1)\wr_r\BZ_p$.  Then $G$ is $FSZ_p$, and is $FSZ_p^+$ if and only if $p\leq 3$.  As a consequence, for $p>3$ $G$ is an $FSZ$ group which is not $FSZ^+$.
\end{thm}
\begin{proof}
    Let $H=F(p,1)$.  By \citep[Corollary 5.4]{IMM} every group is $FSZ_2^+$ and $FSZ_3^+$, so it suffices to consider the case $p>3$.

    To see why the $FSZ_p^+$ condition necessarily fails, we note that by \cref{lem-part:cents-1,thm:fpj-main} there exists $d\in Z(H)$ such that $C_G((d,1,...,1,i))\cong H\ast \BZ_{p^2}$ is not $FSZ_p$ if $p>3$.

    The final claim follows from the first claim combined with \cref{thm:higher-wreath-FSZ}, so we need only prove the $FSZ_p$ property.  We will do so by using \cref{prop:wreath-condition}, from which it suffices to solve
    \begin{align}\label{eq:mixed-eq}
      x_l^p = \Big(\prod_{s=1}^{p} x_{s}\Big)u_0 = d^n,
    \end{align}
    and show the number of solutions is independent of $n$ when $p\nmid n$.

    By \cref{lem:pj-pows} we may suppose that for some $d_1\in\BN$ with $p\nmid d_1$ that either $d= a_1^{p d_1}$ or $d= a_1^{p d_1}a_p^{-d_1}$.

    Now when $d=a_1^{p d_1}$, by \cref{lem:pj-pows} we must have $x_l\in Q$ for all $l$ and $u_0 \in Q$.  Whence in this case the calculation is reduced to establishing the condition of \cref{prop:wreath-condition} in $Q\wr_r\BZ_p$, which follows from \cref{thm:ab-wreath}.  So we may suppose that $d=a_1^{p d_1} a_p^{-d_1}$ with $p\nmid d_1$, and that $x_l\not\in Q$ for all $l$.

    At this point it will be helpful to use additive and vector notation for $Q$; so instead of writing a typical element as $a_1^{n_1}\cdots a_p^{n_p}$, we write it as $(n_1,...,n_p)$.  Let $X=\sum_{t=0}^{p-1} B^t$.  By \cref{lem:x-1} $X(n_1,...,n_p) = (pn_1,0,...,0,-n_1)$, and so in particular the left action of $X$ depends only on the first coordinate modulo $p$.  Note that for all $t\in\BN$ $B^t(n_1,...,n_p)=(n_1,n_2',...,n_p')$ for some $n_2',...,n_p'$ which are $\BZ_p$-linear combinations of the $n_1,...,n_p$ (taken modulo $p$).  Now since $d=(pd_1,0,...,0,-d_1)$, any solution $(x_0,...,x_{p-1})$ to \cref{eq:mixed-eq} must have $x_i = q_i  b^{-t_i}$ for some $p\nmid t_i$ and $q_i\in Q$.  Moreover, writing $u_0 = b^s v_0$ with $v_0\in Q$, by \cref{lem:order-irrev} we see that \cref{eq:mixed-eq} can be equivalently rewritten as the three equations
    \begin{align}
        X(q_i) &= d,\qquad 0\leq i< p;\label{eq:mixed-eq-1}\\
        d-v_0&= \sum_{l=0}^{p-1} B^{\sum_{i=0}^{l-1} t_i} q_l;\label{eq:mixed-eq-2}\\
        \sum_{i=0}^{p-1} t_i &\equiv s \bmod p.\label{eq:mixed-eq-3}
    \end{align}

    For fixed $1\leq t_0,...,t_{p-1}<p$ satisfying \cref{eq:mixed-eq-3} we then consider the group homomorphism $F\colon Q^p\to Q^{p+1}$ defined by \[(y_0,...,y_{p-1})\mapsto (X(y_0),...,X(y_{p-1}),\sum_{l=0}^{p-1} B^{\sum_{i=0}^{l-1} t_i} y_l).\]
    If a solution exists to \cref{eq:mixed-eq} for any $n$ with $p\nmid n$ and with the given $t_i$, then to establish the desired bijection it suffices to show that $(d,...,d)$ is in the image of $F$.

    By bijectivity of the powers of $B$, given $y_0,...,y_{p-1}\in Q$ we can (uniquely) define $r_0,..,r_{p-1}\in Q$ by $r_l = B^{\sum_{i=0}^{l-1} t_i} y_l$ for $0\leq l <p$, and conversely we can obtain the $y_l$ (uniquely) from given $r_l$.  Note that under these relations $X(y_i)=X(r_i)$ for all $i$.  Moreover, we have \[F(y_0,...,y_{p-1}) = (X(r_0),...,X(r_{p-1}),\sum_i r_i).\]  So taking $r_0 = (d_1,0,...,0,-d_1)$ and $r_i = (d_1,0,...,0,0)$ for $1\leq i< p$ we get $(X(r_0),...,X(r_{p-1}),\sum_i r_i) = (d,...,d)$, as desired.

    So we may apply \cref{prop:wreath-condition} as desired to conclude that $G$ is $FSZ_p$, which completes the proof.
\end{proof}

The proof of \cref{thm:fsz-not-plus} yields a reasonably general method for investigating the $FSZ_p$ property of $(A\rtimes\BZ_p)\wr_r\BZ_p$ when $A$ is an (elementary) abelian $p$-group and $A\rtimes\BZ_p$ is $FSZ_p$.  Note that $A\rtimes\BZ_p$ can be non-$FSZ_p$ for some choices of $A$ and the action on it \citep{K16:p-examples}. We demonstrate this with the following result.

\begin{thm}\label{thm:Sp3}
  Let $p$ be a prime.  Then $(\BZ_p\wr_r\BZ_p)\wr_r\BZ_p$ is $FSZ$.  In particular, every Sylow $p$-subgroup of $S_{p^3}$ is $FSZ$.
\end{thm}
\begin{proof}
  The structure of the Sylow subgroups of symmetric groups is well-known; see, for example, \citep{Rot99}.  In the case of $S_{p^3}$ the Sylow $p$-subgroup is isomorphic to $(\BZ_p\wr_r\BZ_p)\wr_r \BZ_p$.

  So let $P=(\BZ_p\wr_r\BZ_p)\wr_r\BZ_p$.  Since $P$ has exponent $p^3$ and \cref{cor:symm-syl-high} implies $P$ is $FSZ_{p^2}$, we need only show that $P$ is $FSZ_p$.

  In the notation of \cref{prop:wreath-condition}, we have $D=\BZ_p\wr_r\BZ_p = \BZ_p^p\rtimes\BZ_p$.  We let $Q=\BZ_p^p$, written in vector notation (as a $\BZ_p$ vector space, in particular).  Let $b$ be an element of $D$ which cyclicly permutes the factors of $Q$, and let $B$ be the matrix which acts on $Q$ from the left which describes right conjugation by $b$.  We have a group endomorphism $J\colon Q\to Q$ given by $J(n_1,...,n_p)=(\sum_i n_i,...,\sum_i n_i)$.  By \citep[Example 4.4]{IMM} this endomorphism describes $p$-th powers of elements $x\in Q\rtimes \BZ_p=\BZ_p\wr_r\BZ_p$ with $x\not\in Q$; all other elements have order dividing $p$.

  Now consider \cref{eq:wreath} for $t=1$.  Then for any solutions to exist for $g\neq 1$ we must have $x_l\not\in Q$ for all $l$.  Moreover, by the properties of $J$ we must have that $d=(t,...,t)\in Q$ for some $t\in \BZ_p$.
  So write $x_l = (n_{l,1},...,n_{l,p})b^{-t_l}$ for each $l$, with $p\nmid t_l$.  Let $u_0=b^s q_0$, with $q_0\in Q$. We fix $t_1,...,t_p$ with $\sum_i t_i = s$ arbitrarily.  As in the proof of the preceding theorem, we are naturally led to consider the group homomorphism $F\colon Q^p\to Q^{p+1}$ defined by
  \[(q_1,...,q_p)\mapsto (J(q_1),...,J(q_p), \sum_{l=1}^{p-1} B^{\sum_{i=0}^{l-1} t_i}q_l).\]
  We can exploit the bijectivity of the powers of $B$, exactly as in the preceding proof, to find $r_1,..,r_p$ for given $q_1,...,q_p$ (or conversely) such that
  \[ F(q_1,...,q_p)= (J(r_1),...,J(r_p),\sum_i r_i).\]
  As in the previous proof, the desired bijection will follow if we show that there exists $(r_1,...,r_p)\in Q^p$ such that $(J(r_1),...,J(r_p),\sum_i r_i)=(d,...,d)$.  Since we have noted that $d= (t,..,t)$, writing $r_i = (m_{i,1},...,m_{i,p})$ the necessary and sufficient conditions for $(d,...,d)$ to be in the image of $F$ is
  \begin{align*}
    \sum_k m_{i,k} = t , \ \mbox{ for all } i;\\
    \sum_k m_{k,i} = t, \ \mbox{ for all } i.
  \end{align*}
  So defining $m_{i,j} = t \delta_{i,j}$, where $\delta_{i,j}$ is the Kronecker delta, we see that $(d,...,d)$ is in the image, as desired.  Thus $P$ is $FSZ_p$ by \cref{prop:wreath-condition}, and this completes the proof.
\end{proof}

\appendix
\section{The non-FSZ groups of order \texorpdfstring{$5^7$}{5\textasciicircum 7}}\label{appendix}
The author has found it useful for investigating non-$FSZ$ groups to have a wide variety of groups for which these properties are already known.  The group $F(p,j)$ above was discovered in this fashion, for example.  In this appendix we will compile lists of all (isomorphism classes of) groups of order $5^7$ which are non-$FSZ$, using the \verb"SmallGroups" library of GAP.  Constructing these lists is relatively simple: simply use the \verb"FSZtest" function of \citet{PS16} (see also \citep{K16:monster}) on groups with suitable exponent and nilpotence class, and then similarly apply this function over each non-$FSZ$ group's maximal subgroups and quotients using the GAP functions \verb"MaximalSubgroups" and \verb"MinimalNormalSubgroups".  As such, we will simply state results.  All non-$FSZ$ groups in this appendix are non-$FSZ_5$.

By the notation $[a,b]$ we mean the set of those integers $n$ with $a\leq n \leq b$. On the other hand, a (non-sequential) set of integers will be denoted by curly braces, for example $\{a,b,c\}$.

\begin{prop}\label{prop:exp125}
  The group $G=\text{SmallGroup}(5^7,n)$ is non-$FSZ$ with exponent $125$ if and only if $n$ is in one of the following:
  \begin{multicols}{2}
  \begin{itemize}
    \item $[656,659]$;
    \item $[663,666]$;
    \item $[676,679]$;
    \item $[714,717]$;
    \item $[720,723]$;
    \item $[733,736]$;
    \item $\{670, 727\}$.
  \end{itemize}
  \end{multicols}
  In particular, there are 26 such groups.  All of them are $FSZ^+_{25}$, and all of their proper subgroups and quotients are $FSZ^+$.
\end{prop}

There is a total of 915 non-$FSZ$ groups of order $5^7$ and exponent $25$, which we break down into classes based on if they have maximal subgroups or quotients which are also non-$FSZ$.

\begin{prop}\label{prop:exp25-bothnew}
  The group $G=\text{SmallGroup}(5^7,n)$ is non-$FSZ$ with exponent $25$ and has all proper subgroups and quotients $FSZ$ if and only if $n$ is in one of the following:
  \begin{itemize}
    \item $[777,783]$;
    \item $[792,799]$;
    \item $[808,832]$;
    \item $[1022, 1264]$.
  \end{itemize}
  In particular, there are 284 such groups.
\end{prop}

\begin{prop}\label{prop:exp25-newbyquot}
  The group $G=\text{SmallGroup}(5^7,n)$ is non-$FSZ$ with exponent $25$ and has all proper quotients $FSZ$ and at least one non-$FSZ$ (maximal) subgroup if and only if $n$ is in one of the following:
\begin{multicols}{2}
  \begin{itemize}
    \item $\{1282,1284,1285,23288\}$;
    \item $[1287,1296]$;
    \item $[1298,1303]$;
    \item $[1305,1369]$;
    \item $[1371,1380]$;
    \item $[23229,23232]$;
    \item $[23244,23248]$;
    \item $[23259,23266]$;
    \item $[23281,23286]$;
    \item $[23300,23307]$;
    \item $[23311,23314]$;
    \item $[23322,23325]$;
    \item $[23331,23314]$;
    \item $[23322,23325]$;
    \item $[23331,23338]$;
    \item $[23341,23348]$;
    \item $[23354,23360]$;
    \item $[23366,23373]$;
    \item $[23378,23402]$;
    \item $[23409,23412]$;
    \item $[23422,23425]$;
    \item $[23434,23437]$;
    \item $[23449,23459]$.
  \end{itemize}
\end{multicols}
  In particular, there are 213 such groups.
\end{prop}

\begin{prop}\label{prop:exp25-newbymax}
  The group $G=\text{SmallGroup}(5^7,n)$ is non-$FSZ$ with exponent $25$ and has all proper subgroups $FSZ$ and at least one non-$FSZ$ quotient if and only if $n$ is one of the values in \cref{table:exp25-newbymax}.
\begin{table}[ht]
\centering
\caption{ID numbers of non-$FSZ$ groups with order $5^7$, exponent 25, a non-$FSZ$ quotient, and all proper subgroups $FSZ$}
\begin{tabular}{CCCCCCCCCC}\label{table:exp25-newbymax}
349& 351& 353& 355& 356& 357& 358& 359& 360& 361\\
362& 363& 364& 365& 366& 367& 368& 369& 370& 371\\
372& 373& 374& 375& 376& 377& 378& 379& 380& 381\\
382& 384& 385& 386& 387& 389& 390& 392& 393& 395\\
396& 398& 399& 400& 401& 402& 403& 404& 405& 406\\
407& 408& 409& 410& 411& 412& 413& 414& 415& 416\\
417& 418& 419& 420& 421& 422& 423& 424& 425& 426\\
427& 428& 429& 430& 431& 432& 433& 434& 435& 436\\
437& 438& 439& 440& 441& 442& 443& 444& 445& 446\\
447& 448& 449& 450& 451& 452& 454& 455& 456& 457\\
459& 460& 461& 462& 464& 465& 466& 468& 469& 470\\
471& 473& 474& 475& 476& 478& 479& 480& 481& 482\\
483& 484& 485& 486& 487& 488& 489& 490& 491& 492\\
493& 494& 495& 496& 497& 498& 499& 500& 501& 502\\
503& 504& 505& 506& 507& 508& 509& 510& 511& 512\\
513& 514& 515& 516& 517& 518& 519& 520& 521& 522\\
523& 524& 525& 526& 527& 528& 529& 531& 532& 533\\
534& 535& 536& 537& 538& 539& 540& 541& 542& 543\\
544& 545& 546& 547& 548& 549& 550& 551& 552& 553\\
554& 555& 556& 557& 558& 559& 560& 561& 562& 563\\
564& 565& 566& 567& 568& 569& 570& 572& 574& 575\\
576& 578& 579& 580& 581& 582& 583& 586& 587& 588\\
589& 590& 591& 592& 593& 594& 596& 597& 598& 599\\
600& 601& 602& 603& 604& 605& 606& 607& 608& 609\\
610& 611& 612& 613& 614& 615& 616& 617& 618& 620\\
621& 622& 623& 624& 625& 627& 650& 651& 655& 661\\
662& 668& 669& 671& 672& 673& 674& 675& 705& 706\\
710& 711& 712& 713& 719& 725& 726& 728& 729& 730\\
731& 732& 764& 765& 766& 767& 770& 771& 772& 773\\
784& 785& 786& 787& 788& 789& 790& 791& 800& 801\\
802& 803& 804& 805& 806& 807& 839& 840& 847& 848\\
855& 856& 863& 864& 871& 872& 879& 880& 882& 883\\
885& 886& 888& 889& 891& 892& 894& 895& 897& 898\\
899& 900& 901& 902& 903& 904& 905& 906& 907& 908\\
909& 910& 912& 913& 914& 915& 917& 918& 919& 920\\
922& 923& 924& 925& 927& 928& 929& 930& 932& 933\\
934& 935& 937& 938& 939& 940& 941& 942& 943& 944\\
945& 946& 947& 948& 949& 950& 951& 952& 953& 954\\
955& 956& 957& 958& 959& 960& & & &
\end{tabular}
\end{table}
  In particular, there are 386 such groups.
\end{prop}

\begin{prop}\label{table:exp25-both-old}
  The group $G=\text{SmallGroup}(5^7,n)$ is non-$FSZ$ with exponent $25$ and has at least one non-$FSZ$ quotient and at least one non-$FSZ$ (maximal) subgroup if and only $G$ has a non-$FSZ$ direct factor (with order necessarily $5^6$).  This is in turn equivalent to $n$ being in one of the following:
  \begin{itemize}
    \item $[23222,23225]$;
    \item $[23236,23240]$;
    \item $[23250,23257]$;
    \item $[23270,23275]$;
    \item $\{23277\}$;
    \item $[23291,23298]$.
  \end{itemize}
  In particular, there are 32 such groups, corresponding to the 32 distinct isomorphism classes of non-$FSZ$ groups of order $5^6$ \citep{IMM}.
\end{prop}
\clearpage
\bibliographystyle{plainnat}
\bibliography{../references}

\end{document}